\documentclass[12pt]{amsart}
\usepackage{amscd}
\usepackage{mathrsfs}
\usepackage{a4wide}
\usepackage{color}
\usepackage{amssymb,amsmath,amsthm,amsfonts,latexsym}
\usepackage[utf8x]{inputenc}
\usepackage{bbm} 

\usepackage{amsmath,amsfonts,amssymb,amsthm} 
\usepackage{tikz}
\usepackage{hyperref}
\usepackage[all]{xy}

\newtheorem{lemma}{Lemma}
\newtheorem{claim}{Claim}

\newtheorem{corollary}{Corollary}
\theoremstyle{definition}
\newtheorem{question}{Question}
\newtheorem{definition}{Definition}
\newtheorem{proposition}{Proposition}
\newtheorem{theorem}{Theorem}
\newtheorem{fact}{Fact}

\newcommand{\Kat}{Kat\v{e}tov}

\newcommand{\I}{\mathcal{I}}
\newcommand{\J}{\mathcal{J}}

\def\leukfrac#1/#2{\leavevmode
               \kern.1em
                \raise.9ex\hbox{\the\scriptfont0 ${}_#1$}
                \hskip -1pt\kern-.1em
                /\kern-.15em\lower.10ex\hbox{\the\scriptfont0 ${}_#2$}}

\makeatletter
\def\diam{\mathop{\operator@font diam}\nolimits}
\makeatother

\begin{document}

\title[No K-minimal  tall Borel ideal]{No minimal tall Borel ideal in the Kat\v etov order}

\author[Greb\'{\i}k]{Jan Greb\'{\i}k${}^1$}

\address{${}^1$ Institute of Mathematics\\
Academy of Sciences of the Czech Republic\\
\v Zitn\'a 609/25  \\
 110 00 Praha 1-Nov\'e M\v esto, Czech Republic}

\email{Greboshrabos@seznam.cz}

\author[Hru\v{s}\'ak]{Michael Hru\v{s}\'ak${}^2$}

\address{${}^2$Centro de Ciencas Matem\'aticas\\
UNAM\\
A.P. 61-3, Xangari, Morelia, Michoac\'an\\
58089, M\'exico}

\email{michael@matmor.unam.mx}

\thanks{{The first-listed author was supported by the GACR project 15-34700L and RVO:
67985840.}{The second-listed author was supported by a PAPIIT grant IN 102311 and CONACyT grant 177758.}}

\date{\today}

\keywords{Borel ideals, Kat\v{e}tov order}

\subjclass{ 03E05, 03E15, 03E17}

\begin{abstract}
Answering a question of the second listed author we show that there is no tall Borel ideal minimal among all tall Borel ideals in the Kat\v etov order. 
\end{abstract}

\maketitle

\section{Introduction}

Given two ideals $\I$ and $\J$ on $\omega$ we say that $\I$ is \emph{Kat\v etov below} $\J$ ($\I\leq_K\J$) if there is a function 
$f:\omega\to\omega$ such that $f^{-1}[I]\in\J$, for all $I\in\I$. The order, which we refer to as the \emph{Kat\v etov order} was 
introduced by M. Kat\v etov \cite{katetov} in 1968  to study convergence in topological spaces. 

This order  is a usefull tool for studying combinatorial properties of filters and ideals (see e.g. \cite{brendle-flaskova, hrusak-filters,hrusak-garcia,hrusak-zapletal,laczkovichreclaw2008, sabok-zapletal, solecki}. However, many fundamental  problems concerning the  Kat\v etov order  remain  open  \cite{guzman-meza, hrusak-filters, hrusak-borel, meza-tesis, minami-sakai,sakai}. In this note we provide the solution to  possibly the most  fundamental of them all by  proving that there is no tall Borel ideal minimal among all tall Borel ideals in the Kat\v etov order. 

 Our set-theoretic notation is standard and follows \cite{kechristeoriadescriptiva} and \cite{kunen} .

\section{$F_\sigma$ ideals}

In this somewhat preliminary section we describe several ways to code $F_\sigma$ ideals. This allows us to speak about complexity of classes of $F_\sigma$ ideals with certain properties such as containing the ideal of finite sets $\mathsf{fin}$, being tall, being above fixed ideal $\mathcal{I}$ etc.
A very useful way of presenting $F_\sigma$ ideals was given by Mazur \cite{mazur}.
Recall that a map $\varphi:\mathcal{P}(\omega)\to [0,\infty]$ is a \emph{lower-semicontinuous submeasure (lcsms)} if 
\begin{itemize}
\item $\varphi(\emptyset)=0$,
\item $A\subseteq B$ implies $\varphi(A)\le \varphi(B)$,
\item $\varphi(A\cup B)\le \varphi(A)+\varphi(B)$,
\item $\varphi(A)=\lim_{n\to\infty}\varphi(A\cap n)$.
\end{itemize}
To each lcsms $\varphi$ naturally corresponds an ideal $Fin(\varphi):=\{A:\varphi(A)<\infty\}$.

\begin{theorem}[Mazur \cite{mazur}]
An ideal $\mathcal{I}$ is $F_\sigma$ if and only if there is lcsms $\varphi$ such that $\mathcal{I}=Fin(\varphi)$.
\end{theorem}

We let $K(2^\omega)$ denote the hyperspace of all compact subsets of $2^\omega$ endowed with the Hausdorff metric. For $K\in K(2^\omega)$ and $n\in\omega$ we define  $\langle K\rangle^n=\{\bigcup_{i\le n} x_i:\ x_i\in K\}$ and $\downarrow K=\{x\in 2^\omega :\ \exists y\in K\ x\subseteq y\}$. Note that the assignment $K\mapsto \downarrow\langle K\rangle^n$ is Borel, in fact, continuous.
We say that $K\in K(2^\omega)$ is a \emph{code} for an $F_{\sigma}$ ideal $\mathcal{I}_{K}=\bigcup_{n\in\omega} \downarrow \langle K\rangle^n$. Note that $\mathcal{I}_K$ is not necessarily proper or contains $fin$. 

\begin{proposition}\label{prop1}
For every $F_\sigma$ ideal $\mathcal{I}$ there is $K\in K(2^\omega)$ such that $\mathcal{I}=\mathcal{I}_K$.
\end{proposition}
\begin{proof}
We find a compact $K\in K(2^\omega)$ such that $\mathcal{I}=\mathcal{I}_K$ and for every infinite $A\in \mathcal{I}$ there is $x\in K$ such that $x=^*A$.
Let $supp(\mathcal{I})=\{n\in\omega:\{n\}\in \mathcal{I}\}$, we may assume that this set is infinite and is increasingly enumerated as $\{a_n\}_{n\in\omega}$.
Fix a sequence of compacts $\{K_n\}_{n\in\omega}$ such that $\mathcal{I}=\bigcup_{n\in\omega} K_n$.
Define a tree $T=\bigcup_{n\in\omega} T_n$ such that $y\in [T_n]$ if and only if there is $x\in K_n$ such that $y\upharpoonright a_n=0$, $y(a_n)=1$ and $y\upharpoonright (\omega\setminus (a_n+1))=x\upharpoonright (\omega\setminus (a_n+1))$.
Then $K=[T]$ has the desired properties.
\end{proof}

We say that $\{c_n\}_{n\in\omega}$ is a \emph{sequence of colorings} of $\omega$ if each $c_n$ is a coloring of $n-$tuples of $\omega$ with two colors i.e. $c_n:[\omega]^n\to 2$ for each $n\in\omega$. We say that $A\subseteq \omega$ is  \emph{simultaneously monochromatic for $\{c_n\}_{n<\omega}$} if $c_n([A]^n)\le 1$ for every $n\in \omega$, and we denote the set of all simultaneously monochromatic sets by $S(\{c_n\})$. One can easily check that $S(\{c_n\})$ is closed and contains all singletons i.e. $\mathcal{I}_{S(\{c_n\})}$ is an $F_\sigma$ ideal that contains $\mathsf{fin}$. 

\begin{claim}
Let $f,g\in \omega^\omega$. Then for every sequence $\{d_n\}_{n\in \omega}$ where $d_n:[\omega]^{f(n)}\to g(n)$ there is a sequence of colorings $\{c_n\}_{n\in\omega}$ with the same infinite simultaneously monochromatic sets.
\end{claim}

\begin{proof}
First we show the following.
For every $n,k\in \omega$ and every coloring $c:[\omega]^n\to k$, there is a finite sequence $\{c_i\}_{i<k}$ where $c_i:[\omega]^n\to 2$ such that $A\subseteq \omega$ is monochromatic for $c$ if and only if it is monochromatic for $\{c_i\}_{i<k}$. 

To prove this define $c_i(a)=0$ if and only if $c(a)\not= i$ for $a\in [\omega]^n$. Let $A\subseteq\omega$ be monochromatic for $c$ and assume that $c(A)=l$ for some $l<k$, then $c_i(A)=0$ for $i\not=l$ and $c_l(A)=1$.
On the other hand if $A$ is monochromatic for all $\{c_i\}_{i<k}$ then pick arbitrary $a\in[A]^n$. We have that $c(a)=l$ for some $l<\omega$ which implies $c_l(A)=1$ i.e. $c(A)=l$.

Now we can use this observation to get a $\{\bar{d}_{n}\}_{n\in\omega}$ such that $\bar{d}_n:[\omega]^{h(n)}\to 2$ and $\{\bar{d}_{n}\}_{n\in\omega}$ has the same simultaneously monochromatic sets as $\{d_{n}\}_{n\in\omega}$. 

To finish the proof we must put $h$ in the correct form.
It can be easily shown that for every $d:[\omega]^n\to 2$ and every $m\ge n$ there is $\bar{d}:[\omega]^m\to 2$ with the same infinite monochromatic sets.
Therefore we may assume that $h$ is increasing and if $Rng(h)\not=\omega$ then we simply add to $\{\bar{d}_n\}_{n<\omega}$ trivial colorings to the missing places.
\end{proof}

\begin{proposition}
Let $\mathcal{I}$ be an $F_\sigma$ ideal that contains $fin$ then there is a sequence of colorings $\{c_n\}_{n<\omega}$ such that $\mathcal{I}=\mathcal{I}_{S(\{c_n\})}$.
\end{proposition}
\begin{proof} First, fix a sequence of colorings $\{r_n\}_{n\in\omega}$ such that for every infinite $A,I\subseteq \omega$ there is $n\in I$ such that $A$ is not monochromatic for $r_n$ (take for example $r_n(\{k_0,...k_{n-1}\})=1$ if and only if $k_0\equiv...\equiv k_{n-1} \mod n$).
Take $K\in K(2^\omega)$ such that $\downarrow K=K$ and $\mathcal{I}_K=\mathcal{I}$. Then we have that $\chi_{\{n\}}\in K$ for every $n\in \omega$. Define a sequence $\{d_n\}_{n\in \omega}$ such that

\begin{itemize}
\item $d_n:[\omega]^n\to 3$ for every $n\in \omega$,
\item $d_n(a)=1$ if $a\in K$,
\item $d_n(a)=0$ if $a\not\in K$ and $r_n(a)=0$,
\item $d_n(a)=2$ if $a\not\in K$ and $r_n(a)=1$.
\end{itemize}
Now let $A\subseteq \omega$ be infinite. If $A\in K$ than by definition $d_n([A]^n)=1$. Assume on the other hand that $A$ is simultaneously monochromatic for $\{d_n\}_{n\in \omega}$. To finish the proof it is enough to realize that if $d_{n+1}([A]^{n+1})=1$ then $d_n([A]^n)=1$ and that $A$ can be monochromatic in color $\{0,2\}$ only for finitely many $d_n$'s by definition of $\{r_n\}$. Then it follows that $A\in K$ and by the previous claim we can find a sequence of colorings $\{c_n\}_{n<\omega}$ such that $\mathcal{I}=\mathcal{I}_K=\mathcal{I}_{S(\{d_n\})}=\mathcal{I}_{S(\{c_n\})}$.
\end{proof}

We have described three ways how to code $F_\sigma$ ideals. Let us summarize this
\begin{itemize}
\item elements of $K(2^\omega)$ code $F_\sigma$ ideals,
\item elements of $\prod_{n<\omega}2^{[\omega]^n}$ code $F_\sigma$ ideals that contain $\mathsf{fin}$,
\item elements of $\mathbb{S}\subseteq (\omega\cup\{\infty\})^{[\omega]^{<\omega}}$ where $\mathbb{S}$ is closed subspace of submeasures on $[\omega]^{<\omega}$ code $F_\sigma$ ideals.
\end{itemize}

There are natural mappings between the spaces of codes derived from the fact that the spaces really code $F_\sigma$ ideals. For example $\{c_n\}_{n\in \omega}\mapsto S(\{c_n\})$ or $K\mapsto \downarrow K\mapsto \{c_n\}_{n\in \omega}$ as in the proof of Proposition \ref{prop1}.

\section{No \Kat \ minimal tall Borel ideal}

In this section we prove  the main result of the paper.
First we start with some definitions and notation. Given a Polish space $X$, we say that $Y\subseteq X$ is \emph{$\Pi^1_2$-complete} if for every $\Pi^1_2$ subset $Y'$ of a zero-dimensional Polish space $X'$ there is a continuous function $\sigma:X'\to X$ such that $x\in Y'$ if and only if $\sigma(x)\in Y$. For $X,Y$ compact Polish spaces we denote by $C(X,Y)$ the Polish space of all continuous maps from $X$ to $Y$ endowed with the supremum metric.

\begin{theorem}\cite[Theorem 37.14]{kechristeoriadescriptiva}
The set $\mathcal{U}\subseteq C(2^\omega\times 2^\omega,2)^\omega\times 2^\omega$, where $((f_n)_{n\in \omega},x)\in \mathcal{U}$ if and only if there is $A\subseteq \omega$ such that $(f_n(x,\_))_{n\in A}$ converges to $0$ pointwise, is $C(2^\omega\times2^\omega,2)^\omega$-universal for $\Sigma^1_2(2^\omega)$, i.e. every $\Sigma^1_2$ subset of $2^\omega$ appears as a section of $\mathcal U$.
\end{theorem}

Recall that an ideal $\mathcal{J}$ on $\omega$ is \emph{tall} if for every infinite $A\subseteq \omega$ there is an infinite $B\in \mathcal{J}$ such that $B\subseteq A$. Let $\mathcal{T}\subseteq K(2^\omega)$ be the set of codes of tall $F_\sigma$ ideals, i.e $K\in \mathcal{T}$ if $\mathcal{I}_K$ is tall and $\downarrow K$ contains all singletons (or equivalently $\mathcal{I}_K$ is tall and contains $fin$). The following is essentially proved in \cite{beckeretal}, though presented in a different language.

\begin{proposition}
The set $\mathcal{T}$ is $\Pi^1_2$-complete.
\end{proposition}
\begin{proof}
One can verify that $\mathcal{T}$ and $\mathcal{T}'=\{K:\mathcal{I}_K$ is tall$\}$ are $\Pi^1_2$. The map $F:K(2^\omega)\to K(2^\omega)$ defined as $F(K)=K\cup \bigcup_{n<\omega}\chi_{\{n\}}\cup \chi_{\emptyset}$ is continuous and it is a reduction from $\mathcal{T}'$ to $\mathcal{T}$. It is therefore enough to prove that $\mathcal{T}'$ is $\Pi^1_2$-complete.

To that end let $X\subseteq2^\omega$ be $\Sigma^1_2$.
By universality of $\mathcal U$, there is $(f_n)_{n\in \omega}\in C(2^\omega\times2^\omega,2)^\omega$ such that $\{x:((f_n)_{n\in \omega},x)\in \mathcal{U}=X$. 
Define a map $\sigma:2^\omega\to K(2^\omega)$ by $y\in \sigma(x)$ if and only if there is a $z\in 2^\omega$ such that $y(n)=f_n(x,z)$.
This function is clearly continuous.

Given $x\in X$, there is an $A\subseteq \omega$ such that $(f_n(x,\_))_{n\in A}$ converges pointwise to $0$, that is
 $|A\cap y|<\omega$ for every $y\in \sigma(x)$, hence $\sigma(x)\not\in \mathcal{T}'$.

On the other hand if $x\not\in X$ then for every $A\subseteq \omega$, $(f_n(x,\_))_{n\in A}$ does not converge pointwise to $0$.
Therefore there is a $y\in \sigma(x)$ such that $|A\cap y|=\omega$, so $\sigma(x)\in \mathcal{T}'$.
\end{proof}

\begin{proposition}
Let $\mathcal{J}$ be a Borel ideal on $\omega$. The set $\mathcal{J}_{\le_K}=\{K\in K(2^\omega):\mathcal{J}\le_K \mathcal{I}_K\}$ is $\Sigma^1_2$.
\end{proposition}
\begin{proof}
Given $n\in\omega$, let $R_n\subseteq K(2^\omega)\times \omega^\omega\times \mathcal{P}(\omega)$ be the relation defined by $(K,f,A)\in R_n$ if and only if $A\in \mathcal{J} \implies f^{-1}[A] \in \downarrow \langle K\rangle^n$.

\begin{claim}
 $R_n$ is Borel for every $n\in\omega$.
\end{claim}
\begin{proof}
The map $\psi:K(2^\omega)\times \omega^\omega\times \mathcal{P}(\omega)\to K(2^\omega)\times \mathcal{P}(\omega)$ defined by 
$$\psi(K,f,A)=(\downarrow \langle K\rangle^n,f^{-1}[A])$$ %
is Borel and so is the set $M\subseteq K(2^\omega)\times \mathcal{P}(\omega)$ where $(K,A)\in M$ iff $A\in K$. Then $R_n=\psi^{-1}(M)\cup (K(2^\omega)\times \omega^\omega\times (\mathcal{P}(\omega)\setminus \mathcal{J}))$.
\end{proof}
Using the claim we have that

$$\mathcal{J}_{\le_K}=\{K\in K(2^\omega):\exists f\in \omega^\omega \forall A\in \mathcal{P}(\omega) \ (K,f,A)\in \bigcup_{n\in \omega} R_n\}$$
which is clearly $\Sigma^1_2$.
\end{proof}

\begin{theorem}\label{main-thm}
There is no minimal tall Borel  ideal in the Kat\v etov order $\le_K$.
\end{theorem}
\begin{proof}
If there were a minimal tall $\mathcal{J}$ in $\le_K$ then $\mathcal{J}_{\le_K}=\mathcal{T}$. Therefore $\mathcal{T}$ is simultaneously $\Sigma^1_2$, and $\Pi^1_2$-complete which is a contradiction.
\end{proof}

In fact, we have proved more:

\begin{corollary}\label{cor} 
For every tall analytic ideal $\mathcal{J}$ there is a $F_\sigma$ ideal $\mathcal{I}$ such that $\mathcal{J}\not\le_K \mathcal{I}$. 
\end{corollary}

\begin{proof} To see this let $P\subseteq \mathcal{P}(\omega)\times \omega^\omega$ be closed such that $A\in \mathcal{J}$ if there is $g\in \omega^\omega$ such that $P(A,g)$. Then as in Proposition3 we can define $R_n\subseteq K(2^\omega)\times \omega^\omega\times \mathcal{P}(\omega) \times \omega^\omega$ where $(K,f,A,g)\in R_n$ if $\neg P(A,g) \vee f^{-1}[A]\in \downarrow \langle K\rangle^n$. This is also Borel and 
$$\mathcal{J}_{\le_K}=\{K\in K(2^\omega):\exists f\in \omega^\omega \forall A\in \mathcal{P}(\omega) \forall g\in \omega^\omega \ (K,f,A,g)\in \bigcup_{n\in \omega} R_n\}$$
which is again $\Sigma^1_2$.
\end{proof}

\section{Ideals and Selectors}

In the previous section we proved that no tall Borel ideal can be minimal in the Kat\v etov order, but the reason was that the complexity of ``codes" of tall $F_\sigma$ ideals above it, is different from the complexity of all ``codes" of all $F_\sigma$ tall ideals. This can be seen as somewhat unsatisfactory.

In this section we offer a way how to construct for a given tall Borel ideal $\mathcal{J}$ a Borel ideal that is not above it under the assumption of existence of Borel selector - wittness of tallness for $\mathcal{J}$.

\begin{definition}
Let $\mathcal{J}$ be a tall Borel ideal on $\omega$. We say that $\mathcal{J}$ has a \emph{selector} if there is a Borel function $S:\mathcal{P}(\omega)\to\mathcal{P}(\omega)$ such that
\begin{itemize}
\item $S(A)\in\mathcal{J}$,
\item $|S(A)|=A$,
\item $S(A)\subseteq A$.
\end{itemize}
\end{definition}

The question whether every tall Borel ideal has a selector was raised by Uzcategui in \cite{uzcategui}.

\medskip

Let $\mathcal{C}$ be an ideal on $2^{<\omega}$ generated by branches and antichains. It is clear that the set of all branches and antichains is closed and therefore $\mathcal{C}$ is a tall $F_{\sigma}$ ideal. 

\begin{definition}
Given a Borel function  $c:2^{\omega}\to \omega^{\omega}$ we define $\mathcal{I}_c$ to be an ideal on $2^{<\omega}$ generated by
\begin{itemize}
\item antichains in $2^{<\omega}$,
\item $\{x\upharpoonright m: m\in c(x)^{-1}(n)\}$ for every $x\in 2^{\omega}$ and $n\in\omega$,
\item  $A\subseteq \{x\upharpoonright m: m\in\omega\}$ where $x\in 2^{\omega}$, and for every $n\in \omega$ is $|A\cap c(x)^{-1}(n)|=1$.
\end{itemize}
\end{definition}

The constructive version of Theorem \ref{main-thm} is the following:

\begin{theorem}\label{main-2}
Let $\mathcal{I}$ be a Borel ideal with a selector. Then there is a Borel function $c:2^\omega\to\omega^\omega$ such that $\mathcal{I}\not\le_K\mathcal{I}_c$. Moreover, every $\mathcal{I}_c$ is a tall Borel ideal.
\end{theorem}

\begin{fact} \cite[Exercise 18.4]{kechristeoriadescriptiva}
The set $IF_f$ of all finitely splitting trees on $\omega$ is Borel and the map $IF_f\to \omega^{\omega}$ which assigns to a tree $T$ its left-most branch $a_T$ is Borel. 
\end{fact}

\begin{lemma}
Given $K\in K(2^\omega)$, there is a Borel function $\phi:\mathcal{I}_K \to K^{<\omega}$ such that $x\subseteq\phi(x)$.
\end{lemma}
\begin{proof}
Let $x\in 2^\omega$ and $n\in \omega$. We define a tree $T_x\subseteq (2^{<\omega})^n$ where $\langle s_1,...s_n\rangle\in T_x$ if $(x\cap |s_1|)\setminus\bigcup_{i\le n} s_i=\emptyset$ and there are $y_1,...y_n\in K$ such that $s_i\sqsubseteq y_i$. The tree $T_x$ is then finitely splitting  (every node branches to at most $2^n$ nodes) and it is ill-founded if and only if there is $y_1,...y_n\in K$ such that $x\subseteq \bigcup_{i\le n} y_i$. Moreover, the assignment $x\to T_x$ is continuous and since $\mathcal{I}_K=\bigcup_{n\in\omega}\downarrow\langle K\rangle^n$ we may use the previous fact to obtain the desired function $\phi$.
\end{proof}

\begin{lemma}
For every $c:2^{\omega}\to \omega^{\omega}$ is $\mathcal{I}_c$ a tall Borel ideal.
\end{lemma}
\begin{proof}
Let us first note that $\mathcal{I}_c\subseteq \mathcal{C}$, and that each $\mathcal{I}_c$ is tall. Denote the set of all chains and antichains in $2^{<\omega}$ as $\mathbb{B}$ and $\mathbb{A}$ respectively.
\begin{claim}
There is a Borel function $\phi:\mathcal{C}\to \mathbb{B}^{<\omega}\times \mathbb{A}^{<\omega}$ such that $A\subseteq \bigcup\phi(A)$ for every $A\in \mathcal{C}$.
\end{claim}
\begin{proof}
Use the previous lemma. 
\end{proof}

\begin{claim}
For every $k\in\omega$ there is a continuous function $\varphi_k:\omega^\omega\times 2^\omega\to 2^\omega$ such that $\varphi_k(f,x)\subseteq x$ and for all $n\in\omega$ is $|f^{-1}(n)\cap \varphi_k(f,x)|\le k$. 
\end{claim}
\begin{proof}
Let us define the function $\varphi_k$ such that $\varphi_k(f,x)\cap f^{-1}(n)$ are the first $k$ or less elements of $x\cap f^{-1}(n)$ depending on the size of the intersection.
\end{proof}

\begin{claim}
For every $n\in\omega$ the function $\eta_n:\omega^\omega\to 2^\omega$ such that $\eta_n(f)=f^{-1}(\{1,...,n\})$ is continuous.
\end{claim}

\begin{claim}
The function $\Gamma:2^\omega\times \mathcal{P}(\omega)\to \mathcal{P}(2^{<\omega})$ defined as $\Gamma(x,A)=\{s:s\sqsubseteq x,|s|\in A\}$ is continuous.
\end{claim}

Let us dentote as $\phi_0$ and $\phi_1$ the composition of $\phi$ and the projection to $\mathbb{B}^{<\omega}$ and $\mathbb{A}^{<\omega}$ respectively. Let $A\in\mathcal{C}$ and denote $|\phi_0(A)|=l(A)$ then $A\in \mathcal{I}_c$ iff there is $n\in \omega$ such that 

$$A\subseteq \bigcup \phi_1(A)\cup\bigcup_{i<l(A)} \Gamma(\phi_0(A)_i,(\varphi_{n}(c(\phi_0(A)_i),A\cap \phi_0(A)_i)\cup \eta_n(c(\phi_0(A)_i)))).$$
As the right hand side is for fixed $n\in\omega$ clearly a Borel function from $\mathcal{C}$ to $\mathcal{I}_c$ and $\subseteq$ is a closed relation we see that $\mathcal{I}_c$ is Borel as desired.
\end{proof}

\begin{proof}[Proof of Theorem \ref{main-2}]
Let us fix some Borel bijection $B:2^\omega\to \omega^{(2^{<\omega)}}$. Then 
$$\Theta(x):=\{|x\upharpoonright n|:n\in S(B(x)(\{x\upharpoonright n:n\in \omega\}))\}$$
is an infinite set because we may assume that $S\upharpoonright \mathcal{I}=id$ and the assigment is Borel. Now fix any Borel map $H:[\omega]^\omega\to \omega^\omega$ such that $|H(A)^{-1}(n)|<\omega$ and 
$$\lim_n |A\cap H(A)^{-1}(n)|=\infty.$$
Finally put $c(x):=H(\Theta(x))$.

We show that this works. Let $f:2^{\omega}\to \omega$ is given. We can find $x\in 2^{\omega}$ such that $B(x)=f$. Then 
$$\{x\upharpoonright n:n\in S(B(x)(\{x\upharpoonright n:n\in \omega\}))\}\not\in \mathcal{I}_c$$
due to definition of $\mathcal{I}_c$ and $c(x)$. But $S(B(x)(\{x\upharpoonright n:n\in \omega\}))\in \mathcal{I}$ and since $f$ was arbitrary we have that $\mathcal{I}\not\le_K\mathcal{I}_c$.
\end{proof}

\begin{corollary}
If a Borel ideal $\mathcal{I}$ is $\le_K$ smaller then all ideals of the form $\mathcal{I}_c$ then $\mathcal{I}$ does not have a selector.
\end{corollary}

Note that the construction in Theorem \ref{main-2} is not optimal, as the ideal $\mathcal I_c$ is very rarely $F_\sigma$, while by Corollary \ref{cor} we know that there is an $F_\sigma$ example.

It can be easily verified that having a selector is upward-closed under $\le_K$. It was observed by Uzcategui \cite{uzcategui} that the proof of the infinite Ramsey theorem gives us that the Random ideal  $\mathcal{R}$ (see \cite{hrusak-borel}), and hence also all ideals Kat\v etov above it, has a selector.

Next we prove that also the ideal introduced by Solecki in \cite{solecki} (for which it is an open problem \cite{hrusak-borel} whether it is Kat\v etov above $\mathcal{R}$) has a Borel selector. The \emph{Solecki ideal} $\mathcal{S}$ is the ideal on the countable set $\Omega:=\{D\subseteq 2^\omega:\mu(D)=\frac{1}{2} \ \& \ D$ is clopen$\}$, where $\mu$ is the product measure. The ideal $\mathcal{S}$ is generated by the sets of the form $\{D\in \Omega:x\in D\}$ where $x\in 2^\omega$. It is known
(see \cite[Lemma 3.3]{hernandez-hernandez-hrusak}) that for any infinite $X\subseteq \omega$ is $\mu(C_X)\ge \frac{1}{2}$ where $C_X:=\{x\in 2^\omega:|\{D\in X:x\in D\}|=\omega\}$.

We shall use the following coding of $G_\delta$ sets. Let $A\subseteq 2^{<\omega}$ and define 
$$X_A=\{x\in 2^\omega:|\{s\in A:x\upharpoonright|s|=s\}|=\omega\}.$$
It can be verified that each $X_A$ is $G_\delta$, and that each $G_\delta$ can be coded in such a way (see e.g.\cite{hrusak-zapletal}). Moreover, we have the following.

\begin{claim}
The map $\varphi_n:\mathcal{P}(2^{<\omega})\to \mathcal{P}(2^{<\omega})$ is continuous where $\varphi_n(A)$ is maximal antichain with the property that for every $s\in \varphi_n(A)$ is $|\{t\in A:t\le s\}|=n-1$.
\end{claim}

\begin{claim}
Let $\mathbb{A}\subseteq \mathcal{P}(2^{<\omega})$ be the set of all antichains. The map $\mu:\mathbb{A}\to [0,1]$ defined by $\mu(A)=\mu(\{x:\exists s\in A, s\le x\})$ is Borel.
\end{claim}
\begin{proof}
It is easy to check that $\mu^{-1}((a,1])$ is open for every $a\in [0,1]$.
\end{proof}

\begin{claim}
Let $\epsilon>0$.
There is a Borel map $\rho_{\epsilon}:\mathbb{A}\to\mathbb{A}$ such that $|\rho(A)|<\omega$, $\mu(\{x:\exists s\in A\setminus \rho_{\epsilon}(A),s\le x\})<\epsilon$ and $\rho_{\epsilon}(A)\subseteq A$.
\end{claim}
\begin{proof}
The map $\beta:\mathbb{A}\to \prod_{s_n\in 2^\omega} [0,1]$ where $\beta(A)(s_n)=\frac{1}{2^{|s_n|}}$ if $s\in A$ and $0$ otherwise is continuous. Therefore the map $\alpha:\mathbb{A}\to \mathbb{N}$ defined by $\alpha(A)=\min\{m:\mu(A)-\sum_{n<m}(\beta(A)(s_n))<\epsilon\}$ is Borel and so is $\rho_\epsilon(A):=A\cap \{s_n:n<\alpha(A)\}$.
\end{proof}

\begin{lemma}
Let $\epsilon>0$ and denote by $\mathbb{A}_{fin}$ the set of all finite antichains in $2^{\omega}$.
There is a Borel map $\Phi_{\epsilon}:\mathcal{P}(2^{<\omega})\to \prod_n \mathbb{A}_{fin}$ such that 
\begin{itemize}
\item $\Phi_{\epsilon}(A)(n)\subseteq \varphi_n(A)$,
\item $\mu(X_A\setminus \cap_n \{x:\exists s\in \Phi_{\epsilon}(A)(n),s\le x\})<\epsilon$.
\end{itemize}
\end{lemma}
\begin{proof}
Simply put $\Phi_\epsilon(A)(n)=\rho_{\frac{\epsilon}{2^{n+1}}}(\varphi_n(A))$.
\end{proof}

Let us enumerate $\Omega$ and recursively define a function $\alpha:\Omega\to \mathbb{A}_{fin}$ such that
\begin{itemize}
\item for every $x\in D$ there is $s\in \alpha(D)$ such that $x\upharpoonright |s|=s$,
\item for every $x,t\in \alpha(D)$ is $|s|=|t|$,
\item for every $n\in \omega$ is $|s|<|t|$ where $s\in D_n$ and $t\in D_{n+1}$.
\end{itemize}

\begin{theorem}
The Solecki ideal $\mathcal{S}$ has a selector.
\end{theorem}
\begin{proof}
The function $C:[\Omega]^\omega\to \mathcal{P}(2^{<\omega})$ defined as $C(X)=\bigcup_{D\in X} \alpha(D)$ is a continuous function that assigns to each $X$ the code of $C_X$.
By lemma there is a Borel function $\Lambda:[\Omega]^\omega\to \prod_n[\Omega]^{<\omega}$ such that 
\begin{itemize}
\item $\Lambda(X)(n)\subseteq X$,
\item $|\{n:D\in \Lambda(X)(n)\}|<\omega$ for every $D\in \Omega$,
\item $\mu(C_X\setminus \cap_n(\bigcup\Lambda(X)(n)))<\frac{1}{4}$.
\end{itemize}
Simply put $D\in \Lambda(X)(n)$ if there is $s\in \alpha(D)\cap \Phi_{\frac{1}{4}}(C(X))$. The first and third condition are verified easily and so is the second if one realizes that $|\alpha(D)|<\omega$.

Let us define a tree $T_X$ on $\Omega$ where $s\in T_X$ if and only if 
\begin{itemize}
\item $\bigcap_{n\le |s|} s(n)\not=\emptyset$,
\item $s(n)\in \Lambda(X)(n)$.
\end{itemize}
Then $T_X$ is finitely branching because each $\Lambda(X)(n)$ is finite and ill-founded because 
$$\mu(C_X\setminus \cap_n(\bigcup\Lambda(X)(n)))<\frac{1}{4}$$
but $\mu(C_X)\ge\frac{1}{2}$ so there must be $x\in C_X$ such that $x\in \bigcup\Lambda(X)(n)$ for each $n\in \omega$. Therefore we may find  for each $n\in \omega$ some $D\in \Lambda(X)(n)$ which gives us the desired branch. Now we may apply Fact 1 and find a branch $b\in [T_X]$ in a Borel way. Then $\{b(n)\}_{n<\omega}$ is in $\mathcal{S}$ and thanks to condition two in the definition of $\Lambda$ it must be infinite. 
\end{proof}

The ideals of the form $\mathcal{I}_c$ have a Borel selector. Let $A\subseteq 2^{<\omega}$ and define $T_A=\{s:\exists a\in A, s\sqsubseteq a\}$. Then $T_A\in IF_f$ so there we can choose a branch in a Borel way. Then either this branch has infinite intersection with $A$ or we can use it to build an antichain in a Borel fashion.

One can nicely characterize those Borel ideals, for which there is a selector which is continuous. Recall that an ideal $\mathcal I$ is \emph{$\omega$-hitting} if given
a family $\{X_n:n\in\omega\}$ of infinite subsets of $\omega$ there is an $I\in\mathcal I$ such that $I\cap X_n\neq\emptyset$ for every $n\in\omega$.

\begin{proposition}
Let $\mathcal{I}$ be a Borel ideal, then $\mathcal{I}$ has a continuous selector if and only if $\mathcal{I}$ is $\omega$-hitting.
\end{proposition}

\begin{proof} Assume first that an ideal $\I$ is $\omega$-hitting.
In \cite{hrusak-meza-minami} it is shown that a Borel ideal $\mathcal I$ is $\omega$-hitting if and only if there is a partition of $\omega$ into finite intervals $\{I_n:n\in\omega\}$ with the property that every set $A\subseteq\omega$ intersecting each interval $I_n$ in at most one point is in the ideal $\mathcal I$. Fix such a partition $\{I_n:n\in\omega\}$, and let
$$g(A)=\{ \min A\cap I_n: n\in\omega \text{ and } A\cap I_n\neq \emptyset\}.$$
Then $g$ is continuous, and 
$g(a)$ is an infinite subset of $A$ for every infinite $A\subseteq \omega$, and by the above property $g(A)\in\mathcal I$.

Now assume that $\mathcal I$ is not $\omega$-hitting, i.e there is a family
 $\{A_n:n\in\omega\}\subseteq [\omega]^\omega$ such that no element of $\mathcal I$ intersects them all, and let $g:[\omega]^\omega\to [\omega]^\omega$ be a continuous function such that $g(A)\subseteq A$ for every $A\subseteq\omega$. to finish the proof we shall show that there is an $A\subseteq \omega $ such that $g(A)\not\in \mathcal I$. To that end we shall recursively construct a sequence $\{B_n:n\in\omega\}\subseteq [\omega]^\omega$
 and an increasing sequence of positive integers $n_0<m_0< \dots n_i<m_i<n_{i+1}<m_{i+1}<\dots$ such that for every $i\in\omega$
 \begin{enumerate}
 \item $B_0=A_0$, 
 \item $n_i\in g(B)$ for every $B\in [\omega]^\omega$ such that $ B\cap m_i=B_i\cap m_i $, 
 \item $B_{i+1}\cap m_i= B_i$, and  
 \item $B_{i+1}\setminus m_i= A_{i+1}\setminus m_i$.
\end{enumerate}
  Doing this is easy using continuity of $g$. Now, notice that by (3) the sequence $\{B_n:n\in\omega\}$ converges to some  $B\in[\omega]^\omega$, and
  $g(B)\supseteq\{n_i:i\in\omega\}$ by (2) and the continuity of $g$. Then, however, $n_i\in g(B)\cap A_i$ for every $i\in\omega$, hence $g(B)\not\in\mathcal I$ (no set in $\I$ intersects all of the $A_i$'s).
  \end{proof}

Note that there are many tall Borel ideals which are not $\omega$-hitting hence do not admit a continuous selector, for inctance the ideal $\mathsf{nwd}$ of nowhere dense subsets of the rationals, or the ideal $\mathsf{fin}\times\mathsf{fin}$ (see e.g. \cite{hrusak-filters}).

The original question of Uzcategui, however, remains open:

\begin{question}[Uzcategui \cite{uzcategui}]
Does every tall Borel ideal admit a Borel selector?
\end{question}

As noted by Uzcategui, it is not even clear whether every tall $F_\sigma$ ideal admits a Borel selector.

\end{document}